\documentclass[10pt]{amsart}
\usepackage{amssymb, enumitem}
\usepackage[all]{xy}
\usepackage{tikz}

\usepackage{aliascnt}
\usepackage{verbatim}
\usepackage[english]{babel}
\usepackage{enumitem}
\usepackage{mathtools}

	\usepackage{amsmath} %formulars and symbols
	\usepackage{amssymb} %more symbols
	\usepackage{mathtools} %curved arrows
	\usepackage[T1]{fontenc} %if a font-nerd reads this
	\usepackage[allcolors=black]{hyperref} %clickable links
	\usepackage[utf8]{inputenc} %some font stuff??
	\usepackage{bm} %bold symbols

	\usepackage{graphicx} %graphics
	\usepackage{tikz-cd} %commutative diagrams
	\usepackage[noadjust]{cite} %bibliography, noadjust= no silly spaces before citation
	\usepackage{todonotes} %insert todo notes in text
	\usepackage{setspace} %Zeilenabstand Titelseite

	\usepackage{amsthm} %loads the theorem package
	
	\theoremstyle{plain}
		
		\newtheorem{thm}{Theorem}[section]	
		\newtheorem{lem}[thm]{Lemma}
		\newtheorem{prop}[thm]{Proposition}
		
		\newtheorem{cor}[thm]{Corollary}
		\newtheorem{claim}{Claim}
		
		\newtheorem{thmintro}{Theorem}
		 
		\newtheorem{propintro}[thmintro]{Proposition}
		 
		\newtheorem{corintro}[thmintro]{Corollary}

	\theoremstyle{definition}

	\theoremstyle{remark}
		\newtheorem{rem}[thm]{Remark}

		\newcommand{\C}{\mathbb C}

		\DeclareMathOperator{\id}{id}

		\DeclareMathOperator{\inj}{inj}
		\DeclareMathOperator{\alg}{alg}
		\newcommand{\corr}{\mathcal E_{\mathfrak{Corr}}}
		\DeclareMathOperator{\im}{im}

		\newcommand{\acton}{\curvearrowright}

\title{Partial tensor-product functors and crossed-product functors}

\author{Julian Kranz and Timo Siebenand}

\thanks{Both authors were funded by the Deutsche Forschungsgemeinschaft (DFG, German Research Foundation) - Project-ID 427320536 - SFB 1442, as well as by Germany's Excellence Strategy EXC 2044 390685587, Mathematics Münster: Dynamics-Geometry-Structure.}

\begin{document}



	\maketitle

	\begin{abstract}
		For a given discrete group $G$, we apply results of Kirchberg on exact and injective tensor products of $C^*$-algebras to give an explicit description of the minimal exact correspondence crossed-product functor and the maximal injective crossed-product functor for $G$ in the sense of Buss, Echterhoff and Willett. In particular, we show that the former functor dominates the latter. 
	\end{abstract}

\section{Introduction}
A fruitful approach to construct examples of $C^*$-algebras is to complete $*$-algebras with respect to certain $C^*$-norms. 
For instance, if $G\acton A$ is an action of a discrete group on a $C^*$-algebra, one can complete the algebraic crossed product $A\rtimes_{\alg} G$ to get the \emph{maximal crossed product} $A\rtimes G$ or the \emph{reduced crossed product} $A\rtimes_r G$.

In the last decade, there has been an increasing interest in exotic completions of $A\rtimes_{\alg}G$, i.e. completions which strictly lie between the maximal and reduced completion. One important motivation comes from the \emph{Baum--Connes conjecture with coefficients} \cite{baum94} which predicts that the Baum--Connes assembly map
	\[\mu\colon K^G_*(\underline{E}G,A)\to K_*(A\rtimes_r G)\]
is an isomorphism. Counterexamples to the conjecture were constructed in \cite{higson02} by exploiting non-exactness of the functor $-\rtimes_r G$ for certain groups $G$. Later, in \cite{baum15} it was suggested to modify the conjecture by replacing the reduced crossed product with the \emph{minimal exact Morita compatible crossed product}. This modification strictly enlarges the class of actions $G\acton A$ for which the conjecture is known to hold and does not change the statement of the conjecture for exact groups. Other motivations to study \emph{exotic crossed-product functors} come from a-$T$-menability and property $(T)$ \cite{MR3138486} or from non-commutative duality \cite{MR3141810,MR3530866,MR3730739,MR3298718}.

General exotic crossed--product functors and their properties were studied systematically by Buss, Echterhoff and Willett \cite{buss15,buss18,buss2018minimal,buss2019injectivity,buss20}. They introduced the \emph{minimal exact crossed-product functor} $-\rtimes_{\mathcal E}G$, the \emph{minimal exact correspondence crossed-product functor} $-\rtimes_{\corr} G$ (which agrees with the minimal exact  Morita compatible crossed-product functor of \cite{baum15} for separable $G$-$C^*$-algebras \cite[Cor. 8.13]{buss18}) and the \emph{maximal injective crossed-product functor} $-\rtimes_{\inj} G$. All these functors agree with the reduced crossed product for exact groups, but their interrelations for non-exact groups are still unclear. 
In particular, it is unclear whether or not $-\rtimes_{\mathcal E}G$ and $-\rtimes_{\corr}G$ agree. A positive answer to this question would imply that the ``new'' Baum-Connes conjecture of \cite{baum15} agrees with the old conjecture of \cite{baum94} for complex coefficients $A=\C$. 
The aim of this article is to provide an explicit description of $-\rtimes_{\inj}G$ and $-\rtimes_{\corr}G$. We hope that the interplay of the universal properties and the explicit descriptions of these functors turn out useful in the future. Our main ingredient is the following construction by Kirchberg:
\begin{thmintro}[\cite{MR1403994}]\label{thm:kirchberg}
	There is a tensor-product functor $-\otimes_{i,\varepsilon}-$ satisfying the following properties:
	\begin{enumerate}
		\item For every $C^*$-algebra $A$, $A\otimes_{i,\varepsilon}-$ is the minimal exact partial tensor-product functor for $A$.
		\item For every $C^*$-algebra $B$, $-\otimes_{i,\varepsilon}B$ is the maximal injective partial tensor-product functor for $B$.
	\end{enumerate}
	In particular, $-\otimes_{i,\varepsilon}-$ is the unique tensor-product functor which is injective in the first variable and exact in the second variable. Furthermore, $-\otimes_{i,\varepsilon}-$ is functorial for completely positive maps in both variables.
\end{thmintro}

In terms of Kirchberg's tensor product, we can describe $-\rtimes_{\inj}G$ and $-\rtimes_{\corr}G$ as follows:

\begin{thmintro}[{Theorem \ref{thm:embedding1}}]
	Let $G$ be a discrete group and let $A$ be a $G$-$C^*$-algebra. Then there are injective $*$-homomorphisms
	\begin{enumerate}
		\item $A\rtimes_{\inj}G\hookrightarrow (A\rtimes_r G)\otimes_{i,\varepsilon}C^*(G)$
		\item  $A\rtimes_{\corr}G\hookrightarrow C^*_r(G)\otimes_{i,\varepsilon}(A\rtimes G)$
	\end{enumerate}
	given by $a\delta_g\mapsto a\delta_g\otimes \delta_g$ and $a\delta_g\mapsto \delta_g\otimes a\delta_g$ respectively. 
\end{thmintro}

We obtain an even more concrete picture using $G$-injective $G$-$C^*$-algebras (see p.\pageref{defn:injective} for the definition). Note that $G$-injective $G$-$C^*$-algebras are always unital.
\begin{propintro}[{Proposition \ref{thm:embedding2}}]
	Let $G$ be a discrete group, let $A$ be a $G$-$C^*$-algebra and let $I$ be a $G$-injective $G$-$C^*$-algebra (e.g. $I=\ell^\infty(G)$). Then the canonical embedding $A\hookrightarrow A\otimes_{\max} I, a\mapsto a\otimes 1$ induces an injective $*$-homomorphism
		\[A\rtimes_{\corr}G\hookrightarrow (A\otimes_{\max} I)\rtimes G.\]
\end{propintro}
Note that for $I=\ell^\infty(G)$, this provides a positive solution to a question asked in \cite[Question 9.4]{buss18} and \cite[8.2]{baum15}. As an application, we are able to compare $-\rtimes_{\inj}G$ and $-\rtimes_{\corr}G$:
\begin{corintro}[{Corollary \ref{cor:compare}}]
	For any discrete group $G$, we have $-\rtimes_{\inj}G\leq -\rtimes_{\corr}G$ and $C^*_{\inj}(G)=C^*_{\corr}(G)$. 
\end{corintro}
Thus, in order to prove that $-\rtimes_{\inj}G$ and $-\rtimes_{\corr}G$ coincide, it would suffice to construct a crossed-product functor which is both exact and injective.

\subsection*{Acknowledgements}
The authors would like to thank Siegfried Echterhoff for helpful discussions and comments and the anonymous referee for pointing out an error in a previous version of this article. 

\section{Preliminaries}\label{sec:preliminaries}

In this section we fix some terminology regarding crossed-product and tensor-product functors. For definitions and basic properties of crossed products and tensor products we refer to \cite{brownozawa,williams07}.

Let ${}^*\textbf{Alg}$ denote the category of $*$-algebras with $*$-homomorphisms as morphisms and let $C^*\textbf{Alg}$ denote the full subcategory of $C^*$-algebras. For a discrete group $G$, we denote by $C^*\textbf{Alg}_G$ the category of $G$-$C^*$-algebras with $G$-equivariant $*$-homomorphisms. 

Let $\mathcal C$ be a category. A functor $F^{\mu}\colon\mathcal C\to C^*\textbf{Alg}$ is a \emph{$C^*$-completion} of a functor $F\colon\mathcal C\to {}^*\textbf{Alg}$, if for every object $X$ in $\mathcal C$, $F^\mu(X)$ is a $C^*$-completion of $F(X)$ and if for every morphism $f$ in $\mathcal C$, $F^\mu(f)$ is an extension of $F(f)$. We define a partial order on the class of $C^*$-completions of a given functor $F$ by declaring $F^\mu \geq F^\nu$ if for every object $X$ in $\mathcal C$, the identity on $F(X)$ extends to a $*$-homomorphism $F^\mu(X)\to F^\nu(X)$. 

For two $C^*$-algebras $A$ and $B$, we denote by $A\odot B$ the algebraic tensor product, by $A\otimes_{\max}B$ the maximal tensor product and by $A\otimes B$ the minimal tensor product. A \emph{tensor-product functor} $-\otimes_\alpha -$ is a $C^*$-completion of the functor 
	\[-\odot -\colon C^*\textbf{Alg}\times C^*\textbf{Alg}\to {}^*\textbf{Alg}.\]
A \emph{partial tensor-product functor $-\otimes_\alpha B$ for $B$} is a $C^*$-completion of the functor 
	\[-\odot B\colon{C}^*\textbf{Alg}\to {}^*\textbf{Alg}.\]
A partial tensor-product functor $-\otimes_\alpha B$ is 
\begin{enumerate}
	\item called \emph{exact} if it maps exact sequences to exact sequences;
	\item called \emph{injective} if it maps injective $*$-homomorphisms to injective $*$-homomorphisms;
	\item said to have the \emph{cp-map property} if for each completely positive map $\varphi\colon A\to C$, the induced map $\varphi\odot \id_B\colon A\odot B\to C\odot B$ extends to a completely positive map $\varphi\otimes_\alpha \id_B\colon A\otimes_\alpha B\to C\otimes_{\alpha}B$.
\end{enumerate}
Every (partial) tensor-product functor is dominated by the maximal tensor-product functor and dominates the minimal tensor-product functor. For a fixed $C^*$-algebra $B$, the functor $-\otimes_{\max} B$ is exact \cite[Prop.\,3.7.1]{brownozawa} whereas the functor $-\otimes B$ is injective. Both functors have the cp-map property \cite[Thm.\,3.5.3]{brownozawa}.

For a discrete group $G$ and a $G$-$C^*$-algebra $A$, we denote by $A\rtimes_{\alg}G=A[G]$ the algebraic crossed product, by $A\rtimes G$ the maximal crossed product and by $A\rtimes_r G$ the reduced crossed product. A \emph{crossed-product functor} $-\rtimes_\mu G$ is a $C^*$-completion of the algebraic crossed-product functor
	\[-\rtimes_{\alg}G\colon C^*\textbf{Alg}_G\to {}^*\textbf{Alg}\]
which dominates the reduced crossed product. We write $C^*_\mu(G):=\C\rtimes_\mu G$. A crossed-product functor $-\rtimes_\mu G$ is
\begin{enumerate}
	\item called \emph{exact} if it maps exact sequences to exact sequences;
	\item called \emph{injective} if it maps injective $G$-equivariant $*$-homomorphisms to injective $*$-homomorphisms;
	\item said to have the \emph{cp-map property} if for each $G$-equivariant completely positive map $\varphi\colon A\to B$, the induced map $\varphi\rtimes_{\alg}G\colon A\rtimes_{\alg}G\to B\rtimes_{\alg}G$ extends to a completely positive map $\varphi \rtimes_\mu G\colon A\rtimes_\mu G\to B\rtimes_\mu G$. 
\end{enumerate}
Every crossed-product functor is dominated by the maximal crossed product and dominates the reduced crossed product. The maximal crossed product $-\rtimes G$ is exact \cite[Prop.\,4.8]{echterhoff2017crossed} and the reduced crossed product $-\rtimes_r G$ is injective \cite[Lem.\,A.16]{MR2203930}. Both functors have the cp-map property \cite[Lem.\,4.8]{buss2019injectivity}. Every injective crossed-product functor has the cp-map property \cite[Thm.\,4.9]{buss18}. Moreover there is a maximal injective crossed-product functor $-\rtimes_{\inj}G$ \cite[Prop.\,3.5]{buss20} and a minimal exact crossed-product functor with the cp-map property $-\rtimes_{\corr}G$ \cite[Cor.\,8.8]{buss18}.

\begin{rem}
	It was shown in \cite[Thm.\,4.9]{buss18} that a crossed-product functor has the cp-map property if and only if it extends to a functor on the $G$-equivariant correspondence category $\mathfrak{Corr}(G)$ as defined in \cite[Def.\,4.4]{buss18}. Therefore crossed-product functors with the cp-map property are called \emph{correspondence crossed-product functors} in \cite{buss18} and $-\rtimes_{\corr}G$ is called the minimal exact \emph{correspondence} crossed-product functor. 
	One can prove a similar characterization for partial tensor-product functors. %We include a proof in the appendix. 
\end{rem}

A $G$-$C^*$-algebra $I$ is called \emph{$G$-injective}\label{defn:injective} if for every injective $G$-equivariant $*$-homomorphism $\iota \colon A\hookrightarrow B$ and every $G$-equivariant completely positive contractive map $\varphi\colon A\to I$, there is a $G$-equivariant completely positive contractive map $\overline{\varphi}\colon B\to I$ such that $\overline{\varphi}\circ \iota=\varphi$. We say that \emph{$\overline{\varphi}$ extends $\varphi$ along $\iota$}. In this case $I$ is unital since there exists a conditional expectation from the unitization $\tilde I$ onto $I$. 
\section{Exact and injective tensor-product functors}\label{sec:tensor product}
In this section we give a detailed proof of a theorem that was stated in \cite{MR1403994} for convenience of the reader. We need a folklore lemma.

\begin{lem}\label{lem:exact}
	Let 
		\[
			\begin{tikzcd}
				0\arrow[r] &I\arrow[r,"\iota"]\arrow[d,"\varphi_I"]	&A\arrow[d,"\varphi_A"]\arrow[r,"q"]	&B\arrow[d,"\varphi_B"]\arrow[r]	&0\\
								0\arrow[r] &I'\arrow[r,"\iota'"]	&A'\arrow[r,"q'"]	&B'\arrow[r]	&0
			\end{tikzcd}
		\]
	be a commutative diagram of $C^*$-algebras and $*$-homomorphisms. Assume that $\iota$ is an ideal inclusion, that the lower row is exact, and that the vertical maps are non-degenerate inclusions. Then we have $\ker q\subseteq \im(\iota)$.
\end{lem}
\begin{proof}
	Let $x\in \ker q$. By exactness, we find $y\in I'$ such that $\iota'(y)=\varphi_A(x)$. Let $(e_\lambda)_\lambda$ be an approximate unit for $I$. Since $\varphi_I$ is non-degenerate, $(\varphi_I(e_\lambda))_\lambda$ is an approximate unit for $I'$ and thus $\|\varphi_I(e_\lambda)y-y\|\to 0$. We obtain $\|\varphi_A\left(\iota(e_\lambda) x-x\right)\|=\|\iota'\left(\varphi_I(e_\lambda)y - y\right)\|\to 0$. This implies $\|\iota(e_\lambda)x-x\|\to 0$ because $\varphi_A$ is isometric and therefore $x\in \im (\iota)$ since $\iota$ is an ideal inclusion.
\end{proof}

\begin{thm}[\cite{MR1403994}]
	There is a tensor-product functor $-\otimes_{i,\varepsilon}-$ satisfying the following properties:
	\begin{enumerate}
		\item For every $C^*$-algebra $A$, $A\otimes_{i,\varepsilon}-$ is the minimal exact partial tensor-product functor for $A$.
		\item For every $C^*$-algebra $B$, $-\otimes_{i,\varepsilon}B$ is the maximal injective partial tensor-product functor for $B$.
	\end{enumerate}
	In particular, $-\otimes_{i,\varepsilon}-$ is the unique tensor-product functor which is injective in the first variable and exact in the second variable. Furthermore, $-\otimes_{i,\varepsilon}-$ has the cp-property in both variables.
\end{thm}

\begin{proof}
	Let $A$ and $B$ be $C^*$-algebras and let $\iota\colon A\hookrightarrow \mathcal B(H)$ be an embedding into the bounded operators on a Hilbert space. We define 
		\begin{equation}\label{eq:defntensor}
			A\otimes_{i,\varepsilon}B:=\iota\otimes \id_B(A\otimes_{\max}B)\subseteq \mathcal B(H)\otimes_{\max}B.
		\end{equation}
	To show that $-\otimes_{i,\varepsilon}-$ has the desired properties, we verify the following claims:
		\begin{claim}
			Up to canonical isomorphism, the definition of $A\otimes_{i,\varepsilon}B$ is independent of $\iota$.
		\end{claim}
		
			Let $\iota'\colon A\hookrightarrow \mathcal B(H')$ be another embedding. Then by Arveson's extension theorem there exist completely positive contractive maps $\Psi\colon\mathcal B(H)\to \mathcal B(H')$ extending $\iota'$ along $\iota$ and $\Phi\colon\mathcal B(H')\to \mathcal B(H)$ extending $\iota$ along $\iota'$. Then $\Psi\otimes_{\max}\id_B$ and $\Phi\otimes_{\max}\id_B$ restrict to mutually inverse $*$-isomorphisms
				\[\iota\otimes \id_B(A\otimes_{\max}B)\cong \iota'\otimes \id_B(A\otimes_{\max}B).\]
			
		\begin{claim}
			$A\otimes_{i,\varepsilon}B$ is functorial for completely positive maps in both variables. 
		\end{claim}
		
			Functoriality for completely positive maps in $B$ follows immediately from the definition. To see functoriality in $A$, let $\varphi\colon A_1\to A_2$ be a completely positive map and let $\iota_j\colon A_j\hookrightarrow \mathcal B(H_j), j=1,2$ be embeddings. Let $\Psi\colon\mathcal B(H_1)\to \mathcal B(H_2)$ be a completely positive map extending $\iota_2\circ \varphi$ along $\iota_1$. Then $\Psi\otimes_{\max}\id_B$ restricts to a completely positive map $A_1\otimes_{i,\varepsilon}B\to A_2\otimes_{i,\varepsilon}B$ extending the canonical map $\varphi\odot \id_B\colon A_1\odot B\to A_2\odot B$.
			
		\begin{claim}
			The functor $-\otimes_{i,\varepsilon}B$ is the maximal injective partial tensor-product functor for $B$.
		\end{claim}
		
			Let $\varphi\colon A_1\hookrightarrow A_2$ be an injective $*$-homomorphism and let $\iota\colon A_2\hookrightarrow \mathcal B(H)$ be an embedding. Then $\varphi\circ \iota\colon A_1\hookrightarrow \mathcal B(H)$ is an embedding too. Inserting this embedding into \eqref{eq:defntensor} shows that $\varphi\otimes \id_B\colon A_1\otimes_{i,\varepsilon}B\to A_2\otimes_{i,\varepsilon}B$ is isometric and therefore injective. 
			Now let $-\otimes_{\alpha}B$ be another injective partial tensor-product functor for $B$ and let $A\hookrightarrow \mathcal B(H)$ be an embedding. Then the canonical quotient map $\mathcal B(H)\otimes_{\max}B\to \mathcal B(H)\otimes_\alpha B$ restricts to a quotient map $A\otimes_{i,\varepsilon}B\to A\otimes_\alpha B$. Thus, $-\otimes_{i,\varepsilon}B$ is maximal.
		
		\begin{claim}
			The functor $A\otimes_{i,\varepsilon}-$ is exact. 
		\end{claim}
		
			Let $0\to I \xrightarrow{\iota} B\xrightarrow{\pi} Q\to 0$ be an exact sequence of $C^*$-algebras. Assume first that $A$ is unital and choose a unital embedding $A\hookrightarrow \mathcal B(H)$. Then the upper row of the diagram
			 	\[
			 		\begin{tikzcd}
			 			0\arrow[r]	&A\otimes_{i,\varepsilon}I\arrow[d,hook]\arrow[r,"\id_A\otimes \iota"]	&[2em]A\otimes_{i,\varepsilon}B\arrow[d,hook]\arrow[r,"\id_A\otimes \pi"]	&[2em]A\otimes_{i,\varepsilon}Q\arrow[d,hook]\arrow[r]	&0\\
			 			0\arrow[r]	&\mathcal B(H)\otimes_{\max}I\arrow[r,"\id_{\mathcal B(H)}\otimes \iota"]	&\mathcal B(H)\otimes_{\max}B\arrow[r,"\id_{\mathcal B(H)}\otimes \pi"]	&\mathcal B(H)\otimes_{\max}Q\arrow[r]	&0
			 		\end{tikzcd}
			 	\]
			 is exact by Lemma \ref{lem:exact}. Now assume that $A$ is not unital and denote by $\tilde A$ its unitization. By the above, the middle and lower row of the diagram 
			 \begin{equation}\label{eq:3x3Lemma}
			 	\begin{tikzcd}
			 		&0\arrow[d] &0\arrow[d] &0\arrow[d] &~\\
			 		0\arrow[r]		&A\otimes_{i,\varepsilon}I\arrow[d]\arrow[r]		&A\otimes_{i,\varepsilon}B\arrow[d]\arrow[r]		&A\otimes_{i,\varepsilon}Q\arrow[d]\arrow[r] 		&0 \\
			 		0\arrow[r]		&\tilde A\otimes_{i,\varepsilon}I\arrow[d]\arrow[r]		&\tilde A\otimes_{i,\varepsilon}B\arrow[d]\arrow[r]		&\tilde A\otimes_{i,\varepsilon}Q \arrow[d]\arrow[r]		&0 \\
			 		0\arrow[r]		&\C\otimes_{i,\varepsilon}I\arrow[d]\arrow[r]		&\C\otimes_{i,\varepsilon}B\arrow[d]\arrow[r]		&\C\otimes_{i,\varepsilon}Q\arrow[d]\arrow[r] 		&0 \\
			 		&0 &0 &0 &
			 	\end{tikzcd}
			 \end{equation}
			 are exact. Since the extension $0\to A\to\tilde A\to \C\to 0$ splits, the columns of \eqref{eq:3x3Lemma} are exact as well. Now exactness of the upper row of \eqref{eq:3x3Lemma} follows from the $3\times 3$-Lemma.
		
			\begin{claim}
				The functor $A\otimes_{i,\varepsilon}-$ is the minimal exact partial tensor-product functor. 
			\end{claim}
			
				Let $A\otimes_\alpha -$ be another exact partial tensor-product functor and fix a $C^*$-algebra $B$. Assume first that $B$ is unital and pick a surjective $*$-homomorphism $C^*(F_X)\to B$ where $F_X$ denotes the free group on a set $X$ of unitaries generating $B$. Denote by $I$ the kernel of $C^*(F_X)\to B$ and choose an embedding $\iota\colon A\hookrightarrow \mathcal B(H)$. By  \cite[Cor. 1.2]{MR1282196} (see also \cite{pisier96}), there is a unique $C^*$-norm on $\mathcal B(H)\odot C^*(F_X)$. In particular, we have a canonical $*$-homomorphism
					\[A\otimes_\alpha C^*(F_X)\to A\otimes C^*(F_X)\xrightarrow{\iota\otimes \id}\mathcal B(H)\otimes C^*(F_X)=\mathcal B(H)\otimes_{\max}C^*(F_X)\]
				mapping $A\otimes_\alpha I$ to $\mathcal B(H)\otimes_{\max}I$. By exactness of both $A\otimes_\alpha-$ and $\mathcal B(H)\otimes_{\max}-$, we can fill the following diagram with the dashed $*$-homomorphism $\psi$:
					\[
						\begin{tikzcd}
							0\arrow[r]	&A\otimes_\alpha I\arrow[d]\arrow[r]	&A\otimes_\alpha C^*(F_X)\arrow[r]\arrow[d]	&A\otimes_{\alpha}B\arrow[r]\arrow[d,dashed,"\psi"]	&0\\
							0\arrow[r]	&\mathcal B(H)\otimes_{\max} I \arrow[r]	&\mathcal B(H)\otimes_{\max}C^*(F_X)\arrow[r]	&\mathcal B(H)\otimes_{\max}B\arrow[r]	&0
						\end{tikzcd}
					\]
				By definition, we have $\psi(A\otimes_\alpha B)=A\otimes_{i,\varepsilon}B$. If $B$ is a non-unital $C^*$-algebra, we can apply the same argument to its unitization and use exactness to produce a canonical quotient map $A\otimes_{\alpha}B\to A\otimes_{i,\varepsilon}B$. This proves maximality.
%		\begin{claim}	
%			$-\otimes_{i,\varepsilon}-$ is the unique tensor-product functor which is injective in the first and exact in the second variable.
%		\end{claim}
%	
%	If $-\otimes_\beta-$ another functor with these properties, then we have $-\otimes_\beta \geq -\otimes_{i,\varepsilon}-$ by exactness and $-\otimes_{i,\varepsilon}-\geq -\otimes_\beta-$ by injectivity. Thus we have $-\otimes_\beta-=-\otimes_{i,\varepsilon}-$.
\end{proof}

\begin{rem}
	Let $F$ be a non-amenable free group and $H$ an infinite-dimensional Hilbert space. Then the flip isomorphism $\mathcal B(H)\odot C^*_r(F)\cong C^*_r(F)\odot \mathcal B(H)$ does \emph{not} extend to an isomorphism $\mathcal B(H)\otimes_{i,\varepsilon}C^*_r(F)\cong C^*_r(F)\otimes_{i,\varepsilon}\mathcal B(H)$. Therefore, Kirchberg's tensor-product functor $-\otimes_{i,\varepsilon}-$ is not symmetric. Indeed, we have $C^*_r(F)\otimes_{i,\varepsilon}\mathcal B(H)=C^*_r(F)\otimes \mathcal B(H)$ since $C^*_r(F)$ is exact and $\mathcal B(H)\otimes_{i,\varepsilon}C^*_r(F)=\mathcal B(H)\otimes_{\max}C^*_r(F)$ by construction. But the identity map on $\mathcal B(H)\odot C^*_r(F)$ does not extend to an isomorphism $\mathcal B(H)\otimes_{\max}C^*_r(F)\cong \mathcal B(H)\otimes C^*_r(F)$ since $C^*_r(F)$ does not have the local lifting property \cite[Cor. 3.7.12, Thm. 13.1.6, Cor. 13.2.5]{brownozawa}.
\end{rem}

\section{Application to crossed products}\label{sec:crossed product}
Throughout this section, let $G$ be a \emph{discrete} group. We recall a version of Fell's absorption principle from \cite{MR4273185}. 

\begin{prop}[{\cite[Prop. 2.8]{MR4273185}}]
	Let $-\rtimes_\mu G$ be a crossed-product functor with the cp-map property and let $A$ be a $C^*$-algebra equipped with the trivial $G$-action. Then the canonical map $A\odot C^*_\mu(G)\to A\rtimes_\mu G$ is injective. In particular, $A\mapsto A\otimes_\mu C^*_\mu(G):=A\rtimes_\mu G$ is a partial tensor-product functor for $C^*_\mu(G)$. 
\end{prop}

Although only stated for $\rho=\max$ in \cite{MR4273185}, the proof of the following lemma works verbatim for every crossed-product functor $-\rtimes_\rho G$:
\begin{lem}[{\cite[Lem 2.10]{MR4273185}}]\label{prop:Fell}
	Let $-\rtimes_\mu G$ be a crossed-product functor with the cp-map property and let $-\rtimes_\rho G$ be any crossed-product functor. Then for every $G$-$C^*$-algebra $A$, there is an injective $*$-homomorphism
		\[A\rtimes_\mu G\hookrightarrow (A\rtimes_\rho G)\otimes_\mu C^*_\mu(G)\]
	given by $a\delta_g\mapsto a\delta_g\otimes \delta_g$ for $a\in A,g\in G$. 
\end{lem}

\begin{thm}\label{thm:embedding1}
	For every $G$-$C^*$-algebra $A$, there are injective $*$-homomorphisms 
		\begin{enumerate}
			\item $A\rtimes_{\inj}G\hookrightarrow (A\rtimes_r G)\otimes_{i,\varepsilon}C^*(G),\quad a\delta_g\mapsto a\delta_g\otimes \delta_g$.
			\item $A\rtimes_{\corr}G\hookrightarrow C^*_r(G)\otimes_{i,\varepsilon}(A\rtimes G),\quad a\delta_g \mapsto \delta_g\otimes a\delta_g$.
		\end{enumerate}
\end{thm}

\begin{proof}
	We first prove the statement for $-\rtimes_{\inj}G$. Denote by $A\rtimes_\alpha G$ the image of $A\rtimes G$ in $(A\rtimes_r G)\otimes_{i,\varepsilon}C^*(G)$ under the map $a\delta_g\mapsto a\delta_g\otimes \delta_g$. Then $-\rtimes_\alpha G$ is an injective crossed-product functor and therefore $-\rtimes_\alpha G\leq -\rtimes_{\inj}G$. On the other hand, Lemma \ref{prop:Fell} gives us an embedding 
			\[A\rtimes_{\inj}G\hookrightarrow (A\rtimes_r G)\otimes_{\inj}C^*_{\inj}(G),\quad a\delta_g \mapsto a\delta_g\otimes \delta_g.\]
		Since $-\otimes_{i,\varepsilon}C^*_{\inj}(G)$ is the maximal injective partial tensor-product functor for $C^*_{\inj}(G)$, we have 
			\[-\otimes_{i,\varepsilon}C^*(G)\geq-\otimes_{i,\varepsilon}C^*_{\inj}(G)\geq -\otimes_{\inj}C^*_{\inj}(G)\]
		and therefore $-\rtimes_\alpha G\geq -\rtimes_{\inj}G$. 
		
		We now prove the statement for $-\rtimes_{\corr}G$. Denote by $A\rtimes_\beta G$ the image of $A\rtimes G$ in $C^*_r(G)\otimes_{i,\varepsilon}(A\rtimes G)$ under the map $a\delta_g\mapsto \delta_g\otimes a\delta_g$. Then $-\rtimes_\beta G$ is an exact crossed-product functor with the cp-map property by Lemma \ref{lem:exact} and therefore $-\rtimes_\beta G\geq -\rtimes_{\corr}G$. On the other hand, Lemma \ref{prop:Fell} gives us an embedding 
			\[A\rtimes_{\corr}G\hookrightarrow C^*_{\corr}(G)\otimes_{\corr}(A\rtimes G).\]
		Since $C^*_{\corr}(G)\otimes_{i,\varepsilon}-$ is the minimal exact partial tensor-product functor for $C^*_{\corr}(G)$, we get 
			\[C^*_r(G)\otimes_{i,\varepsilon}-\leq C^*_{\corr}(G)\otimes_{i,\varepsilon}-\leq C^*_{\corr}(G)\otimes_{\corr}-\]
		and therefore $-\rtimes_\beta G\leq -\rtimes_{\corr}G$.
\end{proof}
\begin{rem}
	The statement of the above theorem remains true if we replace $A\rtimes G$ by $A\rtimes_{\corr}G$, $C^*(G)$ by $C^*_{\corr}(G)$, $A\rtimes_r G$ by $A\rtimes_{\inj}G$, or $C^*_r(G)$ by $C^*_{\inj}(G)$. Indeed, the only properties of the maximal (resp. reduced) crossed product that we used in the proof were exactness (resp. injectivity) and the cp-map property.
\end{rem}

\begin{prop}\label{thm:embedding2}
	Let $I$ be a $G$-injective $G$-$C^*$-algebra and let $A$ be any $G$-$C^*$-algebra. Then the canonical embedding $A\hookrightarrow A\otimes_{\max} I, a\mapsto a\otimes 1$ induces an embedding
		\[A\rtimes_{\corr}G\hookrightarrow (A\otimes_{\max} I)\rtimes G.\]
\end{prop}

\begin{proof}
	Denote by $A\rtimes_\alpha G$ the image of $A\rtimes G$ in $(A\otimes_{\max}I)\rtimes G$ under the map $a\delta_g \mapsto (a\otimes 1) \delta_g$. Then $-\rtimes_\alpha G$ is an exact crossed-product functor with the cp-map property by Lemma \ref{lem:exact} and therefore $-\rtimes_\alpha G\geq -\rtimes_{\corr}G$. It remains to prove the converse inequality. Consider $\mathcal  B(\ell^2(G))$ as a $G$-$C^*$-algebra equipped with the conjugation action of the left regular representation $\lambda\colon G\to \mathcal U(\ell^2(G))$. By $G$-injectivity, there is a $G$-equivariant unital completely positive map $\varphi\colon \mathcal B(\ell^2(G))\to I$. Consider the ``untwisting isomorphism''
		\[\Psi\colon \mathcal B(\ell^2(G))\otimes_{\max}(A\rtimes G)\xrightarrow{\cong} (\mathcal B(\ell^2(G))\otimes_{\max}A)\rtimes G,\quad T\otimes a \delta_g \mapsto T\lambda_{g^{-1}} \otimes a\delta_g\]
	and denote by $\kappa$ the following composition of contractive maps.
%		\[begin{tikzcd}
%		a
%			%A\rtimes_{\corr}G\arrow[r,hook, "Thm. \ref{thm:embedding1}"] 
%			%&C^*_r(G)\otimes_{i,\varepsilon}A\rtimes G\arrow[r,hook, "\lambda \otimes \id"] 
%			%&\mathcal B(\ell^2(G))\otimes_{\max}(A\rtimes G) \arrow[r, "\Psi"] 
%			%&(\mathcal B(\ell^2(G))\otimes_{\max}A)\rtimes G\arrow[r, "(\varphi \otimes \id)\rtimes G"] 
%			%(I\otimes_{\max}A)\rtimes G
%		\end{tikzcd}\]
		\[
			\begin{tikzcd}[column sep=4.5 em]
			A\rtimes_{\corr}G\arrow[r,hook, "\text{Thm.} \ref{thm:embedding1}"] 
			&C^*_r(G)\otimes_{i,\varepsilon}(A\rtimes G)\arrow[r,hook, "\lambda \otimes \id"] 
			&\mathcal B(\ell^2(G))\otimes_{\max}(A\rtimes G) \arrow[d, "\Psi"] \\
			&(I\otimes_{\max}A)\rtimes G &(\mathcal B(\ell^2(G))\otimes_{\max}A)\rtimes G\arrow[l, "(\varphi \otimes \id)\rtimes G"] 
			\end{tikzcd}
		\]
		 A straightforward computation shows that $\kappa(a\delta_g)=(a\otimes 1)\delta_g$ for $a\in A$ and $g\in G$. Thus, we have $\kappa(A\rtimes_{\corr}G)=A\rtimes_\alpha G$ and therefore $-\rtimes_{\corr}G\geq -\rtimes_\alpha G$. 
\end{proof}

\begin{cor}\label{cor:compare}
	For any discrete group $G$, we have $-\rtimes_{\inj}G\leq -\rtimes_{\corr}G$ and $C^*_{\inj}(G)=C^*_{\corr}(G)$. 
\end{cor}
\begin{proof}
	Let $A,I$ be $G$-$C^*$-algebras where $I$ is $G$-injective. The embedding $A\hookrightarrow A\otimes_{\max} I$ induces an embedding $A\rtimes_{\inj}G\hookrightarrow (A\otimes_{\max}I)\rtimes_{\inj}G$. The first statement now follows from Proposition \ref{thm:embedding2}. The second statement follows from the same argument and the fact that $I\rtimes G=I\rtimes_{\inj}G$ \cite[Cor. 3.3]{buss20}.
\end{proof}

 \bibliography{references}{}
	\bibliographystyle{alpha}
\end{document}